\documentclass[9pt]{extarticle} 
\usepackage[ 
  paperwidth=5.88in,paperheight=8.60in,
  textwidth=4.2in,textheight=6.7in]{geometry} 

\usepackage{fancyhdr} 
\usepackage{amsmath, amsthm, amsfonts, amssymb, mathdots}
\usepackage{stmaryrd}
\usepackage{paralist}
\usepackage{mathtools}
\usepackage{tikz}
\usepackage{graphicx}
\usepackage{cleveref}
\newtheorem{theorem}{Theorem}

\newtheorem{lemma}{Lemma}
\newtheorem{corollary}{Corollary}
\newtheorem{proposition}{Proposition}
\theoremstyle{definition}

\newtheorem{example}{Example}
\fancypagestyle{headings}{%
\fancyhf{} 
\fancyfoot[L]{
}

}

\newcommand{\diag}{\textnormal{diag}}
\newcommand{\B}{\textnormal{B}}
\newcommand{\RR}{\mathbb{R}}
\newcommand{\CC}{\mathbb{C}}
\newcommand{\rank}{\textnormal{rank}}

\newcommand{\MM}{\textnormal{\textbf{M}}}
\newcommand{\Gl}{\textnormal{\textbf{Gl}}}

\makeatletter 
\renewcommand{\maketitle}{
\pagestyle{plain}
\vspace*{\baselineskip}
\begin{center}
\MakeUppercase{\small Ralph John de la Cruz} \\ \emph{Institute of Mathematics, University of the Philippines} \\ \emph{Diliman, Quezon City 1105, Philippines} \\ (E-Mail: rjdelacruz@math.upd.edu.ph)\\[0.25cm]
\MakeUppercase{\small Philip Saltenberger} \\ \emph{Institute for Numerical Analysis, TU Braunschweig} \\ \emph{Universit\"atsplatz 2, 38106 Braunschweig, Germany} \\ (E-Mail: philip.saltenberger@tu-bs.de) \\ Corresponding author

\vspace*{3\baselineskip}

\MakeUppercase{\large \@title}
\end{center} \vskip-\baselineskip
}
\makeatother

\usepackage{titlesec} 
\titleformat{\section}[hang]{\scshape}{\thesection. }{0pt}{\centering}[]

\title{Density of diagonalizable matrices in sets of structured matrices defined from indefinite scalar products}

\begin{document}
\maketitle
\begin{abstract}
For an (indefinite) scalar product $[x,y]_B = x^HBy$ for $B= \pm B^H \in \Gl_n(\CC)$ on $\CC^n \times \CC^n$ we show that the set of diagonalizable matrices is dense in the set of all $B$-selfadjoint, $B$-skewadjoint, $B$-unitary and $B$-normal matrices.
\end{abstract}

\section{Introduction}
\label{s:intro}

Whenever $B \in \Gl_n(\CC)$ is some arbitrary (nonsingular) matrix, the function
$$[  \cdot , \cdot]_B: \CC^n \times \CC^n \rightarrow \CC, \; (x,y) \mapsto x^HBy \quad (x^H := \overline{x}^T) $$
defines a nondegenerate sesquilinear form on $\CC^n \times \CC^n$. In case $B$ is not necessarily Hermitian positive definite, such forms are also referred to as indefinite scalar products \cite[Sec.\,2]{Mack07}. In this work, we particularly consider such products for the cases $B=B^H$, $B=-B^H$ (with $B^H$ denoting the conjugate transpose of $B$) and $B^HB = I_n$ (i.e. $B$ is unitary). Indefinite scalar products with $B= \pm B^H$ are called orthosymmetric while those with $B^H = B^{-1}$ are called unitary \cite[Def.\,3.1,\,4.1]{Mack07}.

Several classes of matrices $A \in \MM_n(\CC)$ are naturally related to $[  \cdot , \cdot]_B$:
\begin{enumerate}[(a)]
\item A matrix $J \in \MM_n(\CC)$ is called $B$-selfadjoint if $[Jx,y]_B=[x,Jy]_B$ holds for all $x,y \in \CC^n$. It follows that $x^HJ^HBy = x^HBJy$ holds for all $x,y \in \CC^n$ if $J$ is $B$-selfadjoint. This means we have $J^HB = BJ$,  that is, $J = B^{-1}J^HB$. The set of $B$-selfadjoint matrices is denoted by $\mathbb{J}(B)$.
\item A matrix $L \in \MM_n(\CC)$ is called $B$-skewadjoint if $[Lx,y]_B=[x,-Ly]_B$ holds for all $x,y \in \CC^n$. It follows from this equation as in (a) that $L$ is $B$-skewadjoint if and only if $-L = B^{-1}L^HB$. The set of all $B$-skewadjoint matrices is denoted by $\mathbb{L}(B)$.
\item A matrix $G \in \MM_n(\CC)$ is called $B$-unitary if $[Gx,Gy]_B=[x,y]_B$ holds for all $x,y \in \CC^n$. This means that $x^HG^HBGy = x^HBy$ has to hold for all $x,y \in \CC^n$ and implies that $G^HBG = B$. The set of all $B$-unitary matrices is denoted by $\mathbb{G}(B)$.
\item A matrix $N \in \MM_n(\CC)$ is called $B$-normal if $NB^{-1}N^HB = B^{-1}N^HBN$ holds. The set of all $B$-normal matrices is denoted by $\mathcal{N}(B)$\footnote{Notice that the set $\mathcal{N}(B)$ includes the sets $\mathbb{J}(B), \mathbb{L}(B)$ and $\mathbb{G}(B)$.}.
\end{enumerate}

A great many of problems in control systems theory, matrix equations or differential equations involve matrices from the classes of matrices defined above (see, e.g., \cite{Riccati, Mack07} and the references therein).  The monograph \cite{GoLaRod05} contains a comprehensive treatment and applications for the Hermitian case $B=B^H$.

Assume $B = I_n$ is the $n \times n$ identity matrix. Then the sets of $B$-selfadjoint, $B$-skewadjoint, $B$-unitary and $B$-normal matrices $A \in \MM_n(\CC)$ coincide with the sets of Hermitian ($A=A^H$), skew-Hermitian ($A=-A^H$), unitary ($A^HA=I_n$) and normal ($AA^H=A^HA$) matrices. It is a well-known fact that every matrix belonging to any of these four sets of matrices is semisimple, i.e., diagonalizable \cite{Grone}. 
In case $B \neq I_n$ the situation is different:
for instance, consider the $n \times n$ reverse identity matrix $R_n$ (which is Hermitian but not positive definite) and some basic $n \times n$ Jordan block $J(\lambda)$ for some $\lambda$, i.e.
\begin{equation} J(\lambda) = \begin{bmatrix} \lambda & 1 & & & \\ & \lambda & 1 & & \\ & & \ddots & \ddots & \\ & & & \lambda & 1 \\ & & & & \lambda \end{bmatrix}, \quad R_n = \begin{bmatrix} & &  & & 1 \\ & & & 1 & \\ & & \iddots & & \\ & 1 &  & & \\ 1 & &  & & \end{bmatrix}. \label{equ:rev_identity} \end{equation}
If $\lambda \in \mathbb{R}$, then $J(\lambda) \in \mathbb{J}(R_n)$. However, $J(\lambda)$ is the prototype of a matrix which is not diagonalizable (see also \cite[Ex.\,4.2.1]{GoLaRod05}). Similar examples can be found for the other types of matrices and other $B$. In general, for a (skew)-Hermitian  matrix $B$, a matrix $A \in \MM_n(\CC)$ belonging to $\mathbb{J}(B), \mathbb{L}(B), \mathbb{G}(B)$ or $\mathcal{N}(B)$ need not be semisimple. Thus, given some (skew)-Hermitian $B \in \Gl_n(\CC)$, the following question arises: ``How large is the subset of semisimple matrices in the sets $\mathbb{J}(B), \mathbb{L}(B), \mathbb{G}(B)$ and $\mathcal{N}(B)$?'' The answer to this question is simply ``every matrix is semisimple'' in case $B=I_n$ (and whenever $B$ is Hermitian positive definite), but seems to be unknown for general indefinite or skew-Hermitian $B$.


In this work we consider the question raised above from a topological point of view. Recall that $\MM_{n}(\CC)$ can be considered as a topological space with basis $B_R(A) = \lbrace A' \in \MM_{n}(\CC) \, : \, \Vert A-A' \Vert < R \rbrace$ for $A \in \MM_{n}(\CC)$ and $R \in \RR$, $R > 0$ (see, e.g. \cite[Sec.\,11.2]{Hartfiel}). The sets $\mathbb{J}(B), \mathbb{L}(B), \mathbb{G}(B)$ and $\mathcal{N}(B)$ can thus be interpreted as topological spaces on their own equipped with the subspace topology \cite[Sec.\,1.5]{top}. For instance, a subset $\mathcal{S} \subset \mathcal{N}(B)$ is open (in this subspace topology) if $\mathcal{S}$ is the intersection of $\mathcal{N}(B)$ with some open subset of $\MM_{n}(\CC)$. If a property does not hold for all elements in a topological space, it is usually reasonable to ask whether it holds for a dense subset. Thus, a second question, in this context naturally related to the first one, is ``As not all matrices in $\mathbb{J}(B), \mathbb{L}(B), \mathbb{G}(B)$ or $\mathcal{N}(B)$ are semisimple, is the subset of semisimple matrices in these sets at least dense?''


Since the sets $\mathbb{J}(B), \mathbb{L}(B), \mathbb{G}(B)$ and $\mathcal{N}(B)$ are defined by matrix equations, they are (topologically) closed subsets of $\MM_n(\CC)$ (for instance, $N \in \MM_n(\CC)$ is $B$-normal if and only if it satisfies the equation $NB^{-1}N^HB = B^{-1}N^HBN$). Consequently, although the set of semisimple matrices is dense in $\MM_n(\CC)$ \cite[Cor.\,7.3.3]{Lewis}, there is no direct reasoning why this fact should carry over to $\mathbb{J}(B), \mathbb{L}(B), \mathbb{G}(B)$ or $\mathcal{N}(B)$. In fact, a closed subset of $\MM_n(\CC)$ need not contain any diagonalizable matrix at all (see \cite{Hartfiel} for an example). This work is devoted to the second question raised above and gives a positive answer. Our main result can be stated as follows:
\begin{quote}
\textbf{Main result.}  Let $B = \pm B^H \in \Gl_n(\CC)$. Then the set of diagonalizable matrices is dense in $\mathbb{J}(B)$, $\mathbb{L}(B)$, $\mathbb{G}(B)$ and $\mathcal{N}(B)$.
\end{quote}

Here we will work with the 2-norm of matrices\footnote{Our results hold in the same way for any submultiplicative and unitarily invariant matrix norm.}, i.e. $$\Vert A \Vert_2 := \max_{\Vert x \Vert_2 = 1} \Vert Ax \Vert_2.$$ The density result stated above is thus equivalent to the fact that, for any $A \in \mathbb{J}(B)$, $\mathbb{L}(B), \mathbb{G}(B)$ or $\mathcal{N}(B)$ and any $\varepsilon > 0$ there exists some diagonalizable $A'$ from the same class of matrices such that $\Vert A - A' \Vert_2 < \varepsilon$.  Our motivation for the analysis of density stems from the fact that, for certain $B$ such as $B=R_n$ (with $R_n$ as in \eqref{equ:rev_identity}) or
$$ B = J_{2n} = \begin{bmatrix} 0 & I_n \\ -I_n & 0 \end{bmatrix} \in \MM_{2n}(\RR),$$
 semisimple matrices in $\mathbb{J}(B), \mathbb{L}(B), \mathbb{G}(B)$ or $\mathcal{N}(B)$ can all be transformed to a sparse and nicely structured (canonical) form by a $B$-unitary similarity transformation as shown in \cite{CruzSaltenCanForm}. From this point of view our results imply that, even if the transformation established in \cite{CruzSaltenCanForm} does not exists for some specific matrix $A \in \mathbb{J}(B), \mathbb{L}(B)$, $\mathbb{G}(B)$ or $\mathcal{N}(B)$, it will exist for matrices arbitrarily close to $A$ from the same structure class.

In Section \ref{s:basics} some preliminaries and basic results are given. Section \ref{s:density} gives the proofs for the main results on the density of semisimple matrices in $\mathbb{J}(B), \mathbb{L}(B), \mathbb{G}(B)$ and $\mathcal{N}(B)$. We focus on the case $B=B^H$ in Sections \ref{ss:selfskew} to \ref{ss:normal} and discuss the case $B=-B^H$ in Section \ref{ss:skewHermitian}. Some conclusions are given in Section \ref{s:conclusions}.

\section{Preliminaries and basic Results}
\label{s:basics}

The set of all $j \times k$ matrices over $\mathbb{F} = \RR, \CC$ is denoted by $\MM_{j \times k}(\mathbb{F})$. Whenever $j=k$, we use the short-hand-notation $\MM_k(\mathbb{F}) = \MM_{k \times k}(\mathbb{F})$. The notation $\Gl_k(\mathbb{F})$ refers to the general linear group  over $\mathbb{F}^k$ (i.e. the set of $k \times k$ nonsingular matrices over $\mathbb{F}$). The set of all eigenvalues of a matrix $A \in \MM_k(\CC)$ is called the spectrum of $A$ and is denoted $\sigma(A)$. The multiplicity of an eigenvalue $\lambda \in \CC$ as a root of $\det(A - x I_k)$ is called its algebraic multiplicity and is denoted by $\mathfrak{m}(A, \lambda)$. For each $\lambda \in \sigma(A)$ any vector $v \in \CC^k$ that satisfies $Av = \lambda v$ is called an eigenvector for $\lambda$. The set of all those vectors corresponding to $\lambda \in \sigma(A)$ is called the eigenspace for $\lambda$. It is a subspace of $\CC^k$ and its dimension is called the geometric multiplicity of $\lambda$. The conjugate transpose of a vector/matrix is denoted with the superscript $^H$ while $^T$ is used for the transposition without complex conjugation. A matrix $A \in \MM_n(\CC)$ is called semisimple (or diagonalizable), if there exists some $S \in \Gl_n(\CC)$ such that $S^{-1}AS$ is a diagonal matrix.

Let $B \in \Gl_n(\CC)$. For any $A \in \MM_n(\CC)$ let $A^\star := B^{-1}A^HB$. The matrix $A^\star$ is usually referred to as the adjoint for $A$, cf. \cite[Sec.\,2]{Mack07}. The sets of all $B$-selfadjoint, $B$-skewadjoint, $B$-unitary and $B$-normal matrices as introduced in Section \ref{s:intro} can now be characterized by the equations $A=A^\star$, $A=-A^\star$, $A^\star = A^{-1}$ and $AA^\star = A^\star A$, respectively.
Notice that the mapping $^\star : A \mapsto A^\star$ is $\RR$-linear, that is, for any $A,C \in \MM_n(\CC)$ and any $\alpha, \gamma  \in \RR$ it holds that
$$ \big( \alpha A + \gamma C \big)^\star = \alpha A^\star + \gamma C^\star. $$
Moreover, $(AC)^\star = C^\star A^\star$ holds, so $^\star$ is antihomomorphic.

\begin{lemma} \label{lem:R-subspaces}
Let $B \in \Gl_n(\CC)$. Then the sets $\mathbb{J}(B)$ and $\mathbb{L}(B)$ are $\RR$-subspaces of $\MM_n(\CC)$. Moreover, $\mathbb{G}(B)$ is a subgroup of $\Gl_n(\CC)$.
\end{lemma}

\begin{proof}
The first statement follows immediately since $^\star: A \mapsto A^\star$ is $\RR$-linear.

Now let $G \in \mathbb{G}(B)$. Then, as $G^HBG = B$ holds, $\det(G^H) \det(G) = 1$ follows, so $G$ is nonsingular. Moreover, $G^{-H}BG^{-1} = B$ follows, so $G^{-1} \in \mathbb{G}(B)$. Finally, for $G,F \in \mathbb{G}(B)$, we obtain
$$ \big( FG \big)^HB \big( FG \big) = G^H \big( F^HBF \big) G = G^HBG = B.$$
Therefore, $FG \in \mathbb{G}(B)$. Since $I_n \in \mathbb{G}(B)$, the proof is complete.
\end{proof}

The sets $\mathbb{J}(B)$ and $\mathbb{L}(B)$ are often referred to as the Jordan and Lie algebras (cf. \cite[Sec.\,2]{Mack07}) for $[ \cdot, \cdot]_B$ since $\mathbb{J}(B)$ is closed under the operation $A \odot C := \frac{1}{2}(AC+CA)$ whereas $\mathbb{L}(B)$ is closed under the Lie bracket $[A,C] := AC - CA$. The following Corollary \ref{cor:switch1} shows that $\mathbb{J}(B)$ and $\mathbb{L}(B)$ can never be $\CC$-subspaces of $\MM_n(\CC)$. In fact, multiplication by $i$ gives a possibility to easily switch between the sets $\mathbb{J}(B)$ and $\mathbb{L}(B)$.

\begin{corollary} \label{cor:switch1}
Let $B \in \Gl_n(\CC)$. If $J \in \mathbb{J}(B)$, then $iJ \in \mathbb{L}(B)$. On the other hand, if $L \in \mathbb{L}(B)$, then $iL \in \mathbb{J}(B)$.
\end{corollary}

\begin{proof}
Let $J \in \mathbb{J}(B)$. Notice that $$B^{-1}(iJ)^HB = B^{-1}(-iJ^H)B = -i(B^{-1}J^HB) = -iJ$$ for any $J \in \mathbb{J}(B)$. Thus $iJ \in \mathbb{L}(B)$. On the other hand, if $L \in \mathbb{L}(B)$, then $B^{-1}(iL)^HB = -iB^{-1}L^HB = iL$, so $iL \in \mathbb{J}(B)$.
\end{proof}

The following Lemma \ref{lem:switch2} can easily be verified by a straight forward calculation (see also \cite[Sec.\,1]{Mehl}). Therefore, the proof is omitted.

\begin{lemma} \label{lem:switch2}
Let $B \in \Gl_n(\CC)$ and $A \in \MM_n(\CC)$. Furthermore, let $T \in \Gl_n(\CC)$ and consider
$$ A' := T^{-1}AT \quad \textnormal{and} \quad B' := T^HBT.$$
Then $A$ is $B$-selfadjoint/$B$-skewadjoint/$B$-unitary/$B$-normal if and only if $A'$ is $B'$-selfadjoint/$B'$-skew-adjoint/$B'$-unitary/$B'$-normal.
\end{lemma}

Let $B \in \Gl_n(\CC)$. Then any $A \in \MM_n(\CC)$ can always be expressed as $A = S+K$ with
\begin{equation} S = \frac{1}{2} \left( A + A^\star \right) \in \MM_n(\CC), \qquad K := \frac{1}{2} \left( A - A^\star \right) \in \MM_n(\CC). \label{equ:toeplitz} \end{equation}
Assuming that $B= \pm B^H$, it is easy to see that $S \in \mathbb{J}(B)$ and $K \in \mathbb{L}(B)$. In this case, $A=S+K$ can be interpreted as a $B$-analogue of the Toeplitz decomposition stated in \cite[Thm.~4.1.2]{HoJo}. The next Lemma \ref{lem:projection} shows that, in case $B^HB=I_n$ additionally holds, $S$ and $K$ in \eqref{equ:toeplitz} are the best approximations to $A$ from $\mathbb{J}(B)$ and $\mathbb{L}(B)$ with respect to any unitarily invariant matrix norm. The proof of Lemma \ref{lem:projection} is similar to \cite[Thm.\,2]{FanHoffmann55} (see also \cite[Lem.\,8.3]{PS}).

\begin{lemma} \label{lem:projection}
Let $B = \pm B^H \in \Gl_n(\CC)$ with $B^HB = I_n$ and let $\Vert \cdot \Vert$ be any unitarily invariant matrix norm.
\begin{enumerate}
\item[(a)] Let $A \in \MM_n(\CC)$ and $S := \frac{1}{2}(A+A^\star).$ Then $ S \in \mathbb{J}(B)$ and it holds that $\Vert A - S \Vert \leq \Vert A - C \Vert$ for any other matrix $C \in \mathbb{J}(B)$.
\item[(b)] Let $A \in \MM_n(\CC)$ and $K := \frac{1}{2}(A-A^\star).$ Then $ K \in \mathbb{L}(B)$ and it holds that $\Vert A - K \Vert \leq \Vert A - C \Vert$ for any other matrix $C \in \mathbb{L}(B)$.
\end{enumerate}
\end{lemma}

\begin{proof}
(a) Let $\Vert \cdot \Vert$ be some unitarily invariant matrix norm and let $A \in \MM_n(\CC)$ be given. It follows from $B=\pm B^H$ that $(A^\star)^\star = A$. Therefore, we have $S^\star = \frac{1}{2} (A + A^\star)^\star = \frac{1}{2}(A^\star + A) = S$, so $S \in \mathbb{J}(B)$.

Now let $C \in \mathbb{J}(B)$ be arbitrary. Then
\begin{align*}
\Vert A - S \Vert &= (1/2) \Vert A - A^\star \Vert = \Vert (A-C) - (A^\star - C) \Vert \\
&= (1/2) \Vert (A-C) - (A-C)^\star \Vert \leq (1/2) \big( \Vert A-C \Vert + \Vert (A-C)^\star \Vert \big) \\
&= (1/2) \big( \Vert A-C \Vert + \Vert (A-C)^H \Vert \big) = \Vert A-C \Vert.
\end{align*}
In conclusion we have $\Vert A - S \Vert \leq \Vert A - C \Vert$. The proof for (b) proceeds along the same lines.
\end{proof}

Let
$$J_k(\lambda) := \begin{bmatrix} \lambda & 1 & & & \\ & \lambda & 1 & & \\ & & \ddots & \ddots & \\ & & & \lambda & 1 \\ & & & & \lambda \end{bmatrix} \in \MM_k(\CC)$$
denote a basic Jordan block for $\lambda$ if $\lambda \in \RR$ and define $J_k(\lambda) = J_{\tilde{k}}(\lambda) \oplus J_{\tilde{k}}(\overline{\lambda})$ with $2\tilde{k} = k$ in case $\lambda \in \CC \setminus \RR$. Here and in the following, the notation $\oplus $ is used to denote the direct sum of two matrices, i.e. $X \oplus Y = \textnormal{diag}(X,Y)$. The next two theorems for the case when $B=B^H$ are taken from \cite{GoLaRod05} and will be useful for some proofs in Section \ref{s:density}. For convenience, we omit the subscript $k$ for the Jordan blocks in Theorem \ref{thm:pair-form}.

\begin{theorem}[{\cite[Thm.\,5.1.1]{GoLaRod05}}] \label{thm:pair-form}
Let $B = B^H \in \Gl_n(\CC)$ and $A \in \mathbb{J}(B)$. Then there is some $T \in \Gl_n(\CC)$ such that $T^{-1}AT = J$ and $T^HBT = \widetilde{B}$ where
$$ J = J(\lambda_1) \oplus \cdots \oplus J(\lambda_{\alpha}) \oplus J(\lambda_{\alpha +1}) \oplus \cdots \oplus J(\lambda_{\beta})$$
is a Jordan normal form for $A$ where $\lambda_1, \ldots , \lambda_{\alpha}$ are the real eigenvalues of $A$ and $\lambda_{\alpha +1}, \ldots , \lambda_\beta$ are the nonreal eigenvalues of $A$ from the upper half-plane. Moreover,
$$\widetilde{B} :=  T^HBT = \eta_1 P_1 \oplus \cdots \eta_{\alpha} P_{\alpha} \oplus P_{\alpha+1} \oplus \cdots \oplus P_\beta$$
where, for $k=1, \ldots , \beta$, $P_k = R_p$ is the $p \times p$ reverse identity matrix (see also \eqref{equ:rev_identity}) with $p \times p$ being the size of $J(\lambda_k)$ and $\eta = \lbrace \eta_1, \ldots , \eta_\alpha \rbrace$ is an ordered set of signs $\pm1$.
\end{theorem}

The transformation defined in Theorem \ref{thm:cayley} can be interpreted as a $B$-analogue of the well-known Cayley transform \cite{GolVLoan}.

\begin{theorem}[{\cite[Prop.\,4.3.4]{GoLaRod05}}] \label{thm:cayley}
Let $B = B^H \in \Gl_n(\CC)$ and $A \in \mathbb{J}(B)$. Let $w \in \CC$ be a nonreal number with $w \notin \sigma(A)$ and let $\alpha \in \CC$ with $| \alpha | = 1$. Then
\begin{equation} U = \alpha (A - \overline{w} I_n)(A - wI_n)^{-1} \in \mathbb{G}(B) \label{equ:cayley1} \end{equation}
and $\alpha \notin \sigma(U)$.
Conversely, if $U \in \mathbb{G}(B)$, $| \alpha | = 1$ and $\alpha \notin \sigma(U)$, then for any $w \neq \overline{w}$ we have
\begin{equation} A = (wU - \overline{w} \alpha I_n)(U- \alpha I_n)^{-1} \in \mathbb{J}(B) \label{equ:cayley2} \end{equation}
and $w \notin \sigma(A)$. The formulas \eqref{equ:cayley1} and \eqref{equ:cayley2} are inverse to one other.
\end{theorem}

\subsection{Symmetric Polynomials and 1-Regularity}
\label{ss:polynomials}

A matrix $A \in \MM_n(\CC)$ is called 1-regular\footnote{A 1-regular matrix is sometimes also called nonderogatory.}, if the eigenspace for any $\lambda \in \sigma(A)$ is one-dimensional. In particular, $A$ is 1-regular if and only if there is only one basic Jordan block for each eigenvalue $\lambda$ in the Jordan decomposition of $A$.  This implies that a diagonalizable matrix is 1-regular if and only if all its eigenvalues are (pairwise) distinct. Moreover, it follows directly that 1-regularity is preserved under similarity transformations.

The following Theorem \ref{thm:1-regular} characterizes 1-regular matrices and is of central importance in the further discussion.

\begin{theorem}[{\cite[Prop.\,1.1.2]{WeyrForm}}] \label{thm:1-regular}
Let $A \in \MM_n(\CC)$. Then each of the following statements is equivalent to $A$ being 1-regular.
\begin{enumerate}[(a)]
\item The centralizer of $A$ in $\MM_n(\CC)$ coincides with $\CC[A]$. That is, for all $A' \in \MM_n(\CC)$ commuting with $A$, there is some polynomial $p(x) \in \CC[x]$ such that $A' = p(A)$.
\item The dimension of $\CC[A]$ equals $n$, i.e. the matrices $I_n, A, A^2, \ldots , A^{n-1}$ are linearly independent over $\CC$.
\end{enumerate}
\end{theorem}

The next Proposition \ref{prop:commute} shows that, for any given $A \in \mathbb{J}(B)$ we can always find a 1-regular matrix from $\mathbb{J}(B)$ that commutes with $A$. From now one we confine ourselves to the case $B=B^H$ and postpone the discussion of the situation for $B=-B^H$ to Subsection \ref{ss:skewHermitian}.

\begin{proposition} \label{prop:commute}
Let $B = B^H \in \Gl_n(\CC)$ and $A \in \mathbb{J}(B)$. Then there exists a 1-regular matrix $C \in \mathbb{J}(B)$ such that $A$ and $C$ commute.
\end{proposition}

\begin{proof}
We apply Theorem \ref{thm:pair-form} to $A$ and $B$ (the subscript indicating the size of the Jordan blocks is omitted). So, there exists some $T \in \Gl_n(\CC)$ such that
$$\widetilde{A} := T^{-1}AT =  J(\lambda_1) \oplus \cdots \oplus J(\lambda_{\alpha}) \oplus J(\lambda_{\alpha+1}) \oplus \cdots \oplus J(\lambda_\beta)$$
where $\lambda_1, \ldots , \lambda_{\alpha} \in \RR$ are the real eigenvalues of $A$ and $\lambda_{\alpha +1}, \ldots , \lambda_\beta \in \CC$ are the nonreal eigenvalues of $A$ from the upper half-plane. Moreover,
$$\widetilde{B} :=  T^HBT = \eta_1 P_1 \oplus \cdots \eta_{\alpha} P_{\alpha} \oplus P_{\alpha+1} \oplus \cdots \oplus P_\beta$$
where, for $k=1, \ldots , \beta$, $P_k = R_p$ (the $p \times p$ reverse identity matrix, see \eqref{equ:rev_identity}) with $p \times p$ being the size of $J(\lambda_k)$, and each $\eta_j$ is either $+1$ or $-1$. According to Lemma \ref{lem:switch2} we have $\widetilde{A} \in \mathbb{J}(\widetilde{B})$ since $A \in \mathbb{J}(\B)$. Now let $a_1, \ldots , a_{\alpha} \in \RR$ and $a_{\alpha+1} , \ldots , a_\beta \in \CC \setminus \RR$ be arbitrary and pairwise distinct values and consider
$$ \widetilde{C} := J(a_1) \oplus \cdots \oplus J(a_{\alpha}) \oplus J(a_{\alpha+1}) \oplus \cdots \oplus J(a_\beta)$$
where each $J(a_k)$ has the same size as $J(\lambda_k)$, $k=1, \ldots , \beta$. Observe that $\widetilde{C} \in \mathbb{J}(\widetilde{\B})$ and that $\widetilde{A} \widetilde{C} = \widetilde{C} \widetilde{A}$ holds. Moreover, $\widetilde{C}$ is 1-regular since the values $a_k$ are all distinct. We now apply the reverse transformation to obtain
$$A = T \widetilde{A} T^{-1}, \quad C := T \widetilde{C} T^{-1}, \quad B = T^{-H} \widetilde{B} T^{-1}.$$
Note that now $AC = CA$ holds and that $C \in \mathbb{J}(B)$ (according to Proposition \ref{lem:switch2} since we had $\widetilde{C} \in \mathbb{J}(\widetilde{B})$). Finally, as $\widetilde{C}$ was 1-regular so is $C$.
\end{proof}

Using Corollary \ref{cor:switch1}, the result from Proposition \ref{prop:commute} can easily be extended to $\mathbb{L}(B)$.

\begin{corollary}
Let $B = B^H \in \Gl_n(\CC)$ and $A \in \mathbb{L}(B)$. Then there exists a 1-regular matrix $C \in \mathbb{L}(B)$ such that $A$ and $C$ commute.
\end{corollary}

\begin{proof}
This follows immediately from Corollary \ref{cor:switch1} and Proposition \ref{prop:commute}. Additionally, notice that $C \in \MM_n(\CC)$ is 1-regular if and only if $iC$ is 1-regular.
\end{proof}

A polynomial $p(x_1, \ldots , x_n) \in \CC[x_1, \ldots , x_n]$ in $n \geq 1$ unknowns is called symmetric if
$$ p(x_1, \ldots , x_n) = p(x_{\tau(1)}, \ldots , x_{\tau(n)})$$
holds for all permutations $\tau$ of $1,2, \ldots , n$. The following Theorem \ref{thm:main-symm} will be central in the next section.

\begin{theorem}[{\cite[Prop.\,7.1.10]{WeyrForm}}] \label{thm:main-symm}
Let $p(x_1, \ldots , x_n) \in \CC[x_1, \ldots , x_n]$ be a symmetric polynomial and $f: \MM_n(\CC) \rightarrow \CC$ be a function given by
$$ f(A) = p \big( \lambda_1(A), \lambda_2(A), \ldots , \lambda_n(A) \big) =: p(A)$$
where $\lambda_k(A), k=1, \ldots , n$, denote the eigenvalues of $A$.
Then there is a polynomial $q(x_{11}, x_{12}, \ldots , x_{nn}) \in \CC[x_{11}, x_{12}, \ldots , x_{nn}]$ in $n^2$ unknowns such that
$$ \quad f(A) = q \big( a_{11}, a_{12}, \ldots , a_{nn} \big) \quad \forall \; A = [a_{i,j}]_{i,j} \in \MM_n(\CC).$$
\end{theorem}

\begin{proof}
The proof follows the one from \cite{WeyrForm}. Let $A = [a_{i,j}]_{i,j} \in \MM_n(\CC)$ and let $\chi_A(x) = \det(x I_n - A)$ be the characteristic polynomial of $A$. Assume that
\begin{equation} \chi_A(x) = x^n+  c_1x^{n-1} + \cdots + c_{n-1}x + c_n. \label{equ:sympol_1} \end{equation}
for some $c_k \in \CC, k=1, \ldots , n$. It is well known that $c_1, \ldots , c_n$ are polynomials in the entries $a_{ij}, 1 \leq i,j \leq n$, of $A$.

Furthermore, let $\lambda_1, \ldots , \lambda_n \in \CC$ denote the eigenvalues of $A$. Then
$ \chi_A(x) = (x-\lambda_1)(x- \lambda_2) \cdots (x- \lambda_n)$ and its expansion gives
\begin{equation} \chi_A(x) = x^n - s_1 x^{n-1} + s_2 x^{n-2} + \cdots + (-1)^n s_n \label{equ:sympol_2} \end{equation}
for the coefficients $s_1, \ldots , s_n \in \CC$. A closer inspection reveals that $s_1, \ldots , s_n$ are given by the $n$ elementary symmetric polynomials in $\lambda_1, \ldots , \lambda_n$ (see also \cite[Sec.\,2]{Jacob}), that is
\begin{equation} s_1 = \sum_{i} \lambda_i, \;\; s_2 = \sum_{i < j} \lambda_i \lambda_j, \;\; s_3 = \sum_{i < j < k} \lambda_i \lambda_j \lambda_k, \;  \ldots \; , s_n = \prod_{i} \lambda_i. \label{equ:sympol_3} \end{equation}
A comparison of coefficients in \eqref{equ:sympol_1} and \eqref{equ:sympol_2} yields: each elementary symmetric polynomial $s_k = s_k(\lambda_1, \ldots , \lambda_n)$ in \eqref{equ:sympol_3} in the eigenvalues of $A$ agrees with a certain polynomial $c_k = p_k(a_{11}, a_{12}, \ldots , a_{nn})$ in the entries $a_{ij}$ of $A$. In consequence, this is true for any symmetric polynomial $q(\lambda_1, \ldots , \lambda_n)$ since $q$ can always be expressed as a polynomial in $s_1, \ldots , s_n$ (cf. \cite[Thm.\,2.20]{Jacob}).
\end{proof}

\begin{example} \label{ex:main-symm}
We give three applications of Theorem \ref{thm:main-symm} that will be important in the further discussion. Each of these examples can be found in \cite[Sec.~7]{WeyrForm}. First, let $A =[a_{ij}]_{ij} \in \MM_n(\CC)$ with eigenvalues $\lambda_1(A), \ldots , \lambda_n(A)$.
\begin{enumerate}[$(i)$]
\item  The function $p: \MM_n(\CC) \rightarrow \CC$ given by
$$p(A) := p(\lambda_1(A), \ldots , \lambda_n(A)) = \prod_k \lambda_k(A)$$
determines whether $A$ is invertible. That is, $p(A)=0$ if $A$ is singular and $p(A) \neq 0$ otherwise. As $p$ is a symmetric polynomial in $\lambda_1(A), \ldots , \lambda_n(A)$, there is some $q(x_{11}, \ldots , x_{nn})$ is $n^2$ unknowns with $q(A) := q(a_{11}, a_{12}, \ldots , a_{nn}) = p(A)$ for all $A \in \MM_n(\CC)$. The polynomial $q$ in $a_{11}, a_{12}, \ldots , a_{nn}$ is given by the determinant.
\item The function $p: \MM_n(\CC) \rightarrow \CC$ given by
$$p(A) := p(\lambda_1(A), \ldots , \lambda_n(A)) = \prod_{k \neq j} (\lambda_k(A) - \lambda_j(A)),$$
determines whether $A$ has a multiple eigenvalue. That is, $p(A) \neq 0$ if all eigenvalues of $A$ are (pairwise) distinct and $p(A) = 0$ otherwise. Notice that $p(\lambda_1(A), \ldots , \lambda_n(A))$ is a symmetric polynomial in $\lambda_1(A), \ldots , \lambda_n(A)$.
According to Theorem \ref{thm:main-symm}, there exists a polynomial $q \in \CC[x_{11},  \ldots , x_{nn}]$ in $n^2$ unknowns such that
$$q(A) = q(a_{11}, a_{12}, \ldots , a_{nn}) = p(\lambda_1(A), \ldots , \lambda_n(A)) = p(A)$$
for all $A = [a_{ij}]_{ij} \in \MM_n(\CC)$.
\item Now assume $A =[a_{ij}]_{ij} \in \MM_{m \times n}(\CC)$. Recall that a $k \times k$ minor of $A$ is the determinant of a submatrix $A' \in \MM_{k\times k}(\CC)$ obtained from $A$ by deleting $m-k$ rows and $n-k$ columns. Thus (the value of) each minor is expressible as a polynomial in some entries of $A$ according to $(i)$. Now note that for $A$ the fact $\rank(A) \leq r < \min \lbrace m,n \rbrace$ is equivalent to the vanishing of all $(r+1) \times (r+1)$ minors of $A$. As there are $$s := \binom{m}{r+1} \cdot \binom{n}{r+1}$$ such minors of $A$, according to Theorem \ref{thm:main-symm}, there are $s$ polynomials $q_k(x_{11}, \ldots , x_{mn}), 1 \leq k \leq s$, in $mn$ unknowns $x_{11}, x_{12}, \ldots , x_{mn}$ such that
    $ q_k(A) := q_k(a_{11}, a_{12}, \ldots , a_{mn}) = 0$ holds for all $k = 1, \ldots , s$
if and only if $\rank(A) \leq r$. In particular, if $\rank(A) > r$, there is at least one $q_k$ such that $q_k(A) = q_k(a_{11}, \ldots , a_{mn}) \neq 0$.
\end{enumerate}
Do not overlook that, in each case considered in Example \ref{ex:main-symm}, the polynomials that have been determined do not depend on a special matrix.
\end{example}

Recall that, according to Theorem \ref{thm:1-regular}, $A \in \MM_n(\CC)$ is 1-regular if and only if $I_n,A,A^2, \ldots , A^{n-1}$ are linearly independent. In turn this is the case if and only if the matrix
$$ M := \begin{bmatrix} | & | & | & & | \\ \hphantom{x}I_n\hphantom{i} & A & A^2 & \cdots & A^{n-1} \\  | & | & | & & | \end{bmatrix} \in \MM_{n^2 \times n}(\CC),$$
whose columns are $I_n,A,A^2, \ldots$ written in vectorized fashion as $n^2 \times 1$ column vectors, has full rank (i.e. $\rank(M) = n$). Certainly, the entries of $M$ are polynomials in the entries of $A$. The matrix $M$ has rank $\leq n-1$ if and only if all $n \times n$ minors of $M$ vanish simultaneously. Taking Example \ref{ex:main-symm} $(iii)$ into account we have the following result:

\begin{corollary} \label{cor:1-regular}
There exists a collection of $s \geq 1$ (nonzero) polynomials $w_k(x_{11}, \ldots , x_{nn}) \in \CC[x_{11}, x_{12}, \ldots , x_{nn}]$ in $n^2$ unknowns such that $w_k(A) = w_k(a_{11}, \ldots , a_{nn}) = 0$ holds for $A = [a_{ij}]_{ij} \in \MM_n(\CC)$ and all $w_k, k=1, \ldots , s$, if and only if $A$ is not 1-regular.
\end{corollary}

In other words, Corollary \ref{cor:1-regular} states that $w_\ell(A) \neq 0 $ for at least one $\ell, 1 \leq \ell \leq s,$ is sufficient for $A \in \MM_n(\CC)$ to be 1-regular. This fact will be very useful for the proofs in the upcoming section.

\section{Density of diagonalizable Matrices}
\label{s:density}

In this section we prove our main theorems on the density of diagonalizable matrices in the sets $\mathbb{J}(B)$, $\mathbb{L}(B)$, $\mathbb{G}(B)$ and $\mathcal{N}(B)$. For $B=B^H$ the Lie and Jordan algebras are treated in Subsection \ref{ss:selfskew} whereas the set of $B$-unitary matrices and the set of $B$-normal matrices are considered in Subsections \ref{ss:automorph} and \ref{ss:normal}, respectively. The case $B=-B^H$ is considered in Subsection \ref{ss:skewHermitian}.

\subsection{The Lie and Jordan Algebras}
\label{ss:selfskew}
We begin by considering the density of diagonalizable matrices in the sets $\mathbb{J}(B)$ and $\mathbb{L}(B)$ of $B$-selfadjoint and $B$-skewadjoint matrices. The result of Proposition \ref{prop:one-matrix} will be central for the proof of Theorem \ref{thm:dense-subspaces}.

\begin{proposition} \label{prop:one-matrix}
Let $B = B^H \in \MM_n(\CC)$ be nonsingular. Then $\mathbb{J}(B)$ and $\mathbb{L}(B)$ contain a matrix with pairwise distinct eigenvalues.
\end{proposition}

\begin{proof}
Let $(m_{-},m_{+})$ be the inertia\footnote{As $B \in \Gl_n(\CC)$, we have $0 \notin \sigma(B)$. Moreover, as $B=B^H$, all eigenvalues of $B$ are real. In conclusion $B$ has only positive and negative real eigenvalues.} of $B$, that is, denote the number of positive real eigenvalues of $B$ by $m_{+}$ and the number of negative real eigenvalues of $B$ by $m_{-}$. According to a Theorem of Sylvester (cf. \cite[Thm.\,4.5.7]{HoJo}) there exists some $Q \in \Gl_n(\CC)$ such that
$$ B' := Q^HBQ = \begin{bmatrix} -I_{m_{-}} & \\ & I_{m_+} \end{bmatrix}.$$
Now for any $n$ distinct values $\alpha_1, \ldots , \alpha_n \in \RR$ the matrix $$D = \diag(\alpha_1, \ldots , \alpha_n) \in \MM_n(\CC)$$  is in $\mathbb{J}(B')$. Thus $A := QDQ^{-1}$ is in $\mathbb{J}(Q^{-H}(Q^HBQ)Q^{-1}) = \mathbb{J}(B)$ according to Lemma \ref{lem:switch2} and has pairwise distinct eigenvalues. The same statement for $\mathbb{L}(B)$ follows by taking the diagonal entries $i\alpha_1, \ldots , i\alpha_n$ for $D$.
\end{proof}

We now prove the main theorem of this section. Its proof makes use of the fact observed in Example \ref{ex:main-symm} $(ii)$. With the use of Proposition \ref{prop:one-matrix} it would also follow from  \cite[Cor.\,1]{Hartfiel}.

\begin{theorem} \label{thm:dense-subspaces}
Let $B = B^H \in \Gl_n(\CC)$.
\begin{enumerate}[(a)]
\item For any $J \in \mathbb{J}(B)$ and any $\varepsilon > 0$ there exists some diagonalizable $J' \in \mathbb{J}(B)$ such that $\Vert J-J' \Vert_2 < \varepsilon$. Moreover, $J'$ can be chosen to have pairwise distinct eigenvalues.
\item For any $L \in \mathbb{L}(B)$ and any $\varepsilon > 0$ there exists some diagonalizable $L' \in \mathbb{L}(B)$ such that $\Vert L-L' \Vert_2 < \varepsilon$. Moreover, $L'$ can be chosen to have pairwise distinct eigenvalues.
\end{enumerate}
\end{theorem}

\begin{proof}
Let $A = [a_{ij}]_{ij} \in \MM_n(\CC)$. Recall that, according to Example \ref{ex:main-symm} $(ii)$, there is some polynomial $$q(x_{11}, \ldots , x_{nn}) \in \CC[x_{11}, x_{12}, \ldots , x_{nn}]$$ in $n^2$ unknowns $x_{ij}$, $1 \leq i,j \leq n$, such that $q(A) = q(a_{11}, \ldots , a_{nn})=0$ if and only if $A$ has a (i.e. at least one) multiple eigenvalue. Otherwise, that is if all eigenvalues of $A$ are (pairwise) distinct, $q(A) \neq 0$.

(a) Now assume $J \in \mathbb{J}(B)$. According to Proposition \ref{prop:one-matrix}, there is some $E \in \mathbb{J}(B)$ with distinct eigenvalues. Now consider the family of matrices $M(z) := zJ+E$, $z \in \CC$, and the polynomial $\widetilde{q}(z) := q(M(z)) = q(zJ+E) \in \CC[z]$ that only depends on the single variable $z$. Certainly $\widetilde{q}(0) = q(E) \neq 0$ since the eigenvalues of $E$ are distinct. Therefore,  $\widetilde{q} \neq 0$ and $\widetilde{q}$ is not the zero-polynomial. As $\widetilde{q}(z)$ has only a finite number of roots, almost all matrices $M(z_0), z_0 \in \CC$, have distinct eigenvalues. Consequently, the same holds for all $J+cE = c(c^{-1}J+E)$, $c \in \CC$. To guarantee that $J+cE \in \mathbb{J}(B)$, we confine ourselves to the case $c \in \RR$ (see Lemma \ref{lem:R-subspaces}).

Now let $\varepsilon > 0$ be given and choose some $c_0 \in \RR$ with $| c_0 | < \varepsilon / \Vert E \Vert_2$ such that $\widetilde{q}(c_0^{-1}) \neq 0$. Then $J' := J + c_0E \in \mathbb{J}(B)$, $J'$ has $n$ distinct eigenvalues and
$$\Vert J - J' \Vert_2 = | c_0 | \cdot \Vert E \Vert_2 < \frac{\varepsilon}{\Vert E \Vert_2} \cdot \Vert E \Vert_2 = \varepsilon.$$
Moreover, as $J'$ has pairwise distinct eigenvalues it is diagonalizable.

(b) The proof for $L \in \mathbb{L}(B)$ proceeds along the same lines.
\end{proof}

\begin{corollary} \label{cor:denseJL}
Let $B=B^H \in \Gl_n(\CC)$. The set of matrices with pairwise distinct eigenvalues is dense in $\mathbb{J}(B)$ and $\mathbb{L}(B)$.
\end{corollary}

The proof of Theorem \ref{thm:dense-subspaces} also reveals that any matrix from $\mathbb{J}(B)$ or $\mathbb{L}(B)$ can always be expressed as a sum of matrices with pairwise distinct eigenvalues from the same class.

\begin{corollary} \label{cor:sumoftwo}
Let $B=B^H \in \Gl_n(\CC)$. Then every matrix in $\mathbb{J}(B)$ can be expressed as a sum of two diagonalizable matrices from $\mathbb{J}(B)$ with pairwise distinct eigenvalues. The same holds for $\mathbb{L}(B)$.
\end{corollary}

\begin{proof}
Using the notation from the proof of Theorem \ref{thm:dense-subspaces} (a), $E \in \mathbb{J}(B)$ and $c_0 \in \RR$ can be chosen such that $E$ and $J+c_0E$ both have $n$ distinct eigenvalues. Then $J = (J+c_0E) - c_0E$ is a sum of two matrices $J+c_0E, -c_0E \in \mathbb{J}(B)$ which both have pairwise distinct eigenvalues. The proof follows analogously for $\mathbb{L}(B)$.
\end{proof}

Let $A \in \MM_n(\CC)$. Using the decomposition $A = S+K$ with $S = (1/2)(A+A^\star)$ and $K=(1/2)(A-A^\star)$, see \eqref{equ:toeplitz}, accompanied by Corollary \ref{cor:sumoftwo} we end this section with the following observation related to the set $\mathcal{N}(B)$ of $B$-normal matrices.

\begin{proposition} \label{prop:sumoffour}
Let $B = B^H \in \Gl_n(\CC)$. Any matrix $A \in \MM_n(\CC)$ can be expressed as a sum of four matrices from $\mathcal{N}(B)$ with each having pairwise distinct eigenvalues.
\end{proposition}

\begin{proof}
As any matrix from $\mathbb{J}(B)$ and $\mathbb{L}(B)$ can be expressed as a sum of two matrices with $n$ distinct eigenvalues from the same class, Corollary \ref{cor:sumoftwo} can be applied to $A = S+K$ with $S = (1/2)(A+A^\star) \in \mathbb{J}(B)$ and $K=(1/2)(A-A^\star) \in \mathbb{L}(B)$.
\end{proof}

\subsection{The set $\mathbb{G}(B)$ of $B$-unitary matrices}
\label{ss:automorph}

In this section we analyze the set $\mathbb{G}(B)$ and show that it contains a dense subset of diagonalizable matrices. The proof relies on the Cayley transformation (cf. Theorem \ref{thm:cayley}) and uses the result from Theorem \ref{thm:dense-subspaces} (a).

\begin{theorem} \label{thm:dense-automorph}
Let $B = B^H \in \Gl_n(\CC)$. For any $G \in \mathbb{G}(B)$ and any $\varepsilon > 0$ there exists some diagonalizable $G' \in \mathbb{G}(B)$ such that $\Vert G - G' \Vert_2 \leq \varepsilon$.
\end{theorem}

\begin{proof}
Assume $G \in \mathbb{G}(B)$ is given. Let $\alpha \in \CC$, $| \alpha | = 1$, be chosen such that $\alpha \notin \sigma(G)$ and let $w \neq \overline{w}$ be some fixed number. Then, according to Theorem \ref{thm:cayley},
$$ G' := \big( wG - \overline{w} \alpha I_n \big) \big( G - \alpha I_n \big)^{-1} \in \mathbb{J}(B)$$
with $w \notin \sigma(G')$. By Theorem \ref{thm:dense-subspaces} we may construct a sequence $(F_k')_k \in \mathbb{J}(B)$ of diagonalizable matrices with $F_k' \rightarrow G'$ for $k \rightarrow \infty$. Again according to Theorem \ref{thm:cayley}, $F_k := \alpha (F_k' - \overline{w} I_n)(F_k' - wI_n)^{-1} \rightarrow G$ for $k \rightarrow \infty$
since the transformation is continuous (and both transformations in Theorem \ref{thm:cayley} are inverse to each other). Certainly, $w \notin \sigma(F_k')$ has to hold for $F_k$ to be defined. However, $w \notin \sigma(G')$  implies that $w \notin \sigma(F_k')$ will hold if $F_k'$ is close enough to $G'$ (this can be interpreted as a consequence of the Bauer-Fike Theorem, cf. \cite[Thm.\,7.2.2]{GolVLoan}, since all $F_k'$ are diagonalizable). Formally, there is some $\eta >0$ such that  $w \notin \sigma(F_k')$ for all $F_k' \in \MM_n(\CC)$ with $\Vert G' - F_k' \Vert_2 < \eta$.
From now on, it suffices to consider only those $F_k'$ from the sequence which are close enough to $G'$ such that $F_k$ is defined.

Now let $\varepsilon > 0$ be given. Then there exists some $\delta > 0$ such that
$$ \Vert F_k - G \Vert_2 < \varepsilon \quad \textnormal{for all} \; F_k \; \textnormal{such that} \quad \Vert F_k'-G' \Vert_2 < \delta$$
due to the continuity of the transformation.
Now choose some $F_j'$ from the sequence with $\Vert F_j' - G' \Vert_2 < \min \lbrace \delta , \eta \rbrace$ (so, in particular, $w \notin \sigma(F_j')$). Then $F_j \in \mathbb{G}(B)$ is defined and $\Vert F_j - G \Vert_2 < \varepsilon$. As $F_j'$ is diagonalizable, assume $S^{-1}F_j'S = D$ for some diagonal $D \in \MM_n(\CC)$. Then it follows from a direct calculation that
\begin{equation} S^{-1}F_jS = \alpha \big( D - \overline{w} I_n \big) \big( D - w I_n \big)^{-1} \label{equ:automorph1} \end{equation}
is a diagonalization of $F_j$ and the proof is complete.
\end{proof}

\begin{corollary}
Let $B=B^H \in \Gl_n(\CC)$. The set of matrices with pairwise distinct eigenvalues is dense in $\mathbb{G}(B)$.
\end{corollary}

\begin{proof}
This follows from the proof of Theorem \ref{thm:dense-automorph} and Corollary \ref{cor:denseJL}, since the sequence of matrices $(F_k')_k \in \mathbb{J}(B)$ constructed in the proof of Theorem \ref{thm:dense-automorph} can be chosen such that all matrices have pairwise distinct eigenvalues. Then, if $S^{-1}F_j'S = D$ for some diagonal $D \in \MM_n(\CC)$ with $n$ distinct eigenvalues, the matrix $S^{-1}F_jS$ in \eqref{equ:automorph1} has distinct eigenvalues, too.
\end{proof}

\subsection{The set $\mathcal{N}(B)$ of $B$-normal matrices}
\label{ss:normal}

We now consider the density of diagonalizable matrices in the set $\mathcal{N}(B)$ of $B$-normal matrices. The following Lemma \ref{lem:suminN} will be helpful to prove Theorem \ref{thm:main1}. It shows how to construct $B$-normal matrices from pairs of commuting $B$-selfadjoint matrices.

\begin{lemma} \label{lem:suminN}
Let $B = B^H \in \Gl_n(\CC)$. If $F,G \in \mathbb{J}(B)$ and $F$ and $G$ commute, then $A = F \pm iG \in \mathcal{N}(B)$.
\end{lemma}

\begin{proof}
Note that
$$A^\star = (F \pm iG)^\star = B^{-1}(F \pm iG)^HB = B^{-1}F^HB \pm B^{-1}(iG)^HB = F \mp iG$$
so $AA^\star = (F \pm iG)(F \mp iG)$ and $A^\star A = (F \mp iG)(F \pm iG)$. Since $FG = GF$ we see that $AA^\star = A^\star A$ holds.
\end{proof}

Theorem \ref{thm:main1} states that the density result obtained for $\mathbb{J}(B), \mathbb{L}(B)$ and $\mathbb{G}(B)$ before is true for $\mathcal{N}(B)$ under the additional assumption that $B$ is a unitary matrix.

\begin{theorem} \label{thm:main1}
Let $B \in \Gl_n(\CC)$ with $B=B^H$ and $B^HB = I_n$. For any $N \in \mathcal{N}(\B)$ and any $\varepsilon > 0$ there exists some diagonalizable $N' \in \mathcal{N}(\B)$ such that $\Vert N - N' \Vert_2 \leq \varepsilon$.
\end{theorem}

\begin{proof}
Let $N \in \mathcal{N}(B)$ be arbitrary. We define $S := \frac{1}{2}(N + N^\star) \in \mathbb{J}(B)$, $K := \frac{1}{2}(N - N^\star) \in \mathbb{L}(B)$ as in \eqref{equ:toeplitz} and express $N$ as
\begin{equation} N = S + K = S - i^2 K = S - i K_H \label{equ:normal-matrix} \end{equation}
with $K_H  := i K$. Notice that $K_H \in \mathbb{J}(B)$ according to Corollary \ref{cor:switch1}. It follows straight forward that $SK_H = K_HS$ holds, that is
\begin{align*}
S K_H &= (1/2) \big( A + A^\star \big) \cdot (i/2) \big( A - A^\star \big) = (i/4) \big( A+A^\star \big) \big( A - A^\star \big) \\
&= (i/4) \big( A - A^\star \big) \big( A+A^\star \big) = (i/2) \big( A - A^\star \big) \cdot (1/2) \big( A+A^\star \big) \\ &= K_HS,
\end{align*}
so $S$ and $K_H$ commute.

According to Proposition \ref{prop:commute} there exists some 1-regular $E = [e_{ij}]_{ij} \in \mathbb{J}(B)$ such that $K_H$ and $E$ commute, that is, $K_H E = E K_H$ holds.
Now, we consider the family of all matrices $M = M(z) = zS + E \in \MM_n(\CC)$ for $z \in \CC$. As both $S$ and $E$ commute with $K_H$, so does each $M(z)$. In particular, note that $M(0)=E$ is 1-regular.

According to Corollary \ref{cor:1-regular} there is some $w_\ell(x_{11}, \ldots , x_{nn})$ (from the set of polynomials $w_k$, $k=1, \ldots , s$, that vanish simultaneously for matrices that are not 1-regular) such that $w_\ell(E) = w_\ell(e_{11}, e_{12}, \ldots , e_{nn}) \neq 0$. Now, as $S$ and $E$ are fixed, consider $\widetilde{w}(z) := w_\ell(M(z)) = w_\ell(zS+E) \in \CC[z]$ as a single-variable-polynomial and notice that $\widetilde{w}(0) = w_\ell(E) \neq 0$. Thus $\widetilde{w} \neq 0$ is not the zero-polynomial. Recall that $\widetilde{w}(z_0) \neq 0$ is a sufficient condition for $M(z_0) = z_0S+E$ to be 1-regular.
Consequently, as $\widetilde{w}(z)$ does only have a finite number of roots, $M(z_0)=z_0S + E$ will be 1-regular for almost all $z_0 \in \CC$. Therefore $S + cE = c(c^{-1}S + E)$ is also 1-regular for all but a finite number of nonzero $c \in \CC$.

Now let $\varepsilon = 2\epsilon > 0$ be given. Choose some $c \in \RR$ with $|c| \leq \epsilon/(2 \Vert E \Vert_2)$ such that $S_c := S + cE \in \mathbb{J}(B)$ is 1-regular. Then
$$ \Vert S - S_c \Vert_2 = \Vert c E \Vert_2 = |c| \Vert E \Vert_2 \leq \frac{\epsilon}{2 \Vert E \Vert_2} \Vert E \Vert_2 = \frac{\epsilon}{2}.$$
As $S_c$ and $K_H$ commute (recall that $S$ and $E$ both commute with $K_H$) and $S_c$ is 1-regular, there exists some polynomial $p(x) \in \CC[x]$ with $p(S_c) = K_H$ according to Theorem \ref{thm:1-regular} (a). Moreover, from Theorem \ref{thm:dense-subspaces} (see also Corollary \ref{cor:denseJL}), there exists a sequence $(F_k)_k \in \mathbb{J}(B)$, $k \in \mathbb{N}$, of diagonalizable matrices with $F_k \rightarrow S_c$. Thus, $p(F_k) \rightarrow K_H$ for $k \rightarrow \infty$ since $p(x)$ is continuous. Now, for $\epsilon /2$ there exists some $\delta > 0$ such that
\begin{equation} \Vert p(S_c) - p(F_k) \Vert_2 = \Vert K_H - p(F_k) \Vert_2 < \frac{\epsilon}{2} \quad \textnormal{if} \quad \Vert S_c - F_k \Vert_2 < \delta \label{equ:est0} \end{equation}
due to the continuity of $p(x)$. Next, choose some $F_\ell \in \mathbb{J}(B)$ from the sequence $(F_k)_k$ with $\Vert S_c - F_\ell \Vert < \min \lbrace \epsilon/2, \delta \rbrace$. Then we obtain that
\begin{equation} \begin{aligned}
\Vert S - F_\ell \Vert_2 &= \Vert S - S_c + S_c - F_\ell \Vert_2 \\ &\leq  \Vert S - S_c \Vert_2 + \Vert S_c - F_\ell \Vert_2 < \frac{\epsilon}{2} + \frac{\epsilon}{2} = \epsilon.
\end{aligned} \label{equ:est1} \end{equation}
As $p(x) \in \CC[x]$ may have complex coefficients, it might be the case that $p(F_\ell) \notin \mathbb{J}(B)$. However, $G := (1/2)(p(F_\ell) + p(F_\ell)^\star) \in \mathbb{J}(B)$ (recall \eqref{equ:toeplitz}). Assume $p(x)$ is given by $\sum_{k=0}^t a_k x^k$ for complex coefficients $a_k \in \CC, k=0, \ldots , t$. Then
\begin{align*}
p(F_\ell)^\star &= B^{-1} \left( \sum_{k=0}^t a_k F_\ell^k \right)^H B = B^{-1} \left( \sum_{k=0}^t \overline{a_k} \left( F_\ell^k \right)^H \right)B \\ &= B^{-1} \left( \sum_{k=0}^t \overline{a_k} \left( F_\ell^H \right)^k \right) B = \sum_{k=0}^t \overline{a_k} \left( B^{-1}F_\ell^HB \right)^k = \sum_{k=0}^t \overline{a_k} F_\ell^k,
\end{align*}
so $p(F_\ell)^\star = q(F_\ell)$ with $q(x) = \sum_{k=0}^t \overline{a_k} x^k$. In particular, $(1/2)(p(F_\ell) + q(F_\ell)) = r(F_\ell) = \sum_{k=0}^t 2 \mathfrak{Re}(a_k)F_\ell^k$ is a (real) polynomial in $F_\ell$ (where $\mathfrak{Re}(a_k)$ denotes the real part of the complex number $a_k$).
As $$ \Vert p(F_\ell) - \frac{1}{2} (p(F_\ell) + p(F_\ell)^\star ) \Vert_2 = \Vert p(F_\ell) - r(F_\ell) \Vert_2 \leq \Vert p(F_\ell) - X \Vert_2$$ holds for any $X \in \mathbb{J}(B)$ according to Lemma \ref{lem:projection} and $\Vert p(F_\ell) - K_H \Vert_2 < \epsilon /2$ (recall \eqref{equ:est0} and the choice of $F_\ell$) we conclude $\Vert p(F_\ell) - r(F_\ell) \Vert_2 \leq \epsilon /2$. In analogy to \eqref{equ:est1} we obtain
\begin{align*} \Vert K_H - G \Vert_2 &= \Vert K_H - p(F_\ell) + p(F_\ell) - G \Vert_2 \\ &\leq \Vert K_H - p(F_\ell) \Vert_2 + \Vert p(F_\ell) - G \Vert_2 < \frac{\epsilon}{2} + \frac{\epsilon}{2} = \epsilon. \end{align*}
Finally, we arrived at $ \Vert S - F_\ell \Vert_2 < \epsilon $ and $\Vert K_H - r(F_\ell) \Vert_2 < \epsilon $. As $F_\ell, r(F_\ell) \in \mathbb{J}(B)$, note that $N' := F_\ell - i r(F_\ell) \in \mathcal{N}(B)$ according to Lemma \ref{lem:suminN}. Moreover, we have
\begin{align*}
\Vert N - N' \Vert_2 &= \Vert (S -iK_H) - (F_\ell - i r(F_\ell)) \Vert_2 = \Vert (S - F_\ell) - i(K_H - r(F_\ell) \Vert_2 \\
&\leq \Vert S - F_\ell \Vert_2 + \Vert K_H - r(F_\ell) \Vert_2 \leq \epsilon + \epsilon  = 2 \epsilon = \varepsilon.
\end{align*}
As $F_\ell$ is diagonalizable, so is $r(F_\ell)$. In addition, as $F_\ell$ and $r(F_\ell)$ commute, they are simultaneously diagonalizable \cite[Thm.\,1.3.21]{HoJo}. This certainly implies $N' = F_\ell - ir(F_\ell)$ to be diagonalizable (in particular, $N'$ is a polynomial in $F_\ell$). Thus we have found a diagonalizable matrix $N' \in \mathcal{N}(B)$ with distance at most $\varepsilon$ from $N$ and the proof is complete.
\end{proof}

Applying Lemma \ref{lem:switch2} and using the result of Theorem \ref{thm:main1} we may now prove the density of diagonalizable matrices in $\mathcal{N}(B)$ without the assumption of $B$ being unitary.

\begin{theorem} \label{thm:main2}
Let $B \in \Gl_n(\CC)$ with $B=B^H$. For any $N \in \mathcal{N}(\B)$ and any $\varepsilon > 0$ there exists some diagonalizable $N' \in \mathcal{N}(\B)$ such that $\Vert N - N' \Vert_2 \leq \varepsilon$.
\end{theorem}

\begin{proof}
According to Theorem \ref{thm:main1} the statement is true if $B$ is unitary, so assume $B \in \Gl_n(\CC)$ is not unitary. According to \cite[Thm.\,4.5.7]{HoJo} there exists some $Q \in \Gl_n(\CC)$ such that
$$ B' := Q^HBQ = \begin{bmatrix} -I_{m_{-}} & \\ & I_{m_+} \end{bmatrix}$$
where $m_{-}$ ($m_+$) is the number of negative (positive) eigenvalues of $B$. According to Lemma \ref{lem:switch2} we have $Q^{-1} \mathcal{N}(B) Q = \mathcal{N}(B')$. Since $B'$ is unitary, Theorem \ref{thm:main1} applies and the set of diagonalizable matrices is dense in $\mathcal{N}(B')$. As any matrix $A \in \mathcal{N}(B')$ is diagonalizable if and only if $QAQ^{-1} \in \mathcal{N}(B)$ is diagonalizable the density result follows for $\mathcal{N}(B)$.
\end{proof}

In Sections \ref{ss:selfskew}, \ref{ss:automorph} and \ref{ss:normal} we showed the density results of semisimple matrices for $B=B^H$. Hereby, $B$ was an arbitrary (indefinite) Hermitian matrix. Before we pass on to the case $B=-B^H$, do not overlook the special case of $B$ being positive definite (and $[ \cdot, \cdot]_B$ defining a scalar product). Due to Lemma \ref{lem:switch2}, the same situation as for $B=I_n$ takes place and all matrices in $\mathbb{J}(B), \mathbb{L}(B), \mathbb{G}(B)$ and $\mathcal{N}(B)$ are semisimple. In fact, for a Hermitian positive definite matrix $B \in \Gl_n(\CC)$ there exists some $Q \in \Gl_n(\CC)$ such that $Q^HBQ = I_n$. Whenever $A \in \mathbb{J}(B)$, then $A' := Q^{-1}AQ \in \mathbb{J}(I_n)$ is Hermitian and, in consequence, semisimple. As semisimplicity is preserved under similarity transformation, $A$ must have been semisimple, too. The same reasoning holds in an analogous way for $\mathbb{L}(B), \mathbb{G}(B)$ and $\mathcal{N}(B)$ using Lemma \ref{lem:switch2}.

\subsection{Skew-Hermitian Sesquilinear Forms}
\label{ss:skewHermitian}

We now consider the case where $B = - B^H \in \Gl_n(\CC)$ is a skew-Hermitian matrix\footnote{Notice that $n$ needs to be even for $B=-B^H$ to be nonsingular.}. Fortunately, this situation can be completely traced back to the analysis from Sections \ref{ss:selfskew}, \ref{ss:automorph} and \ref{ss:normal}.

First note that, if $B=-B^H$ holds, then  $iB$ is Hermitian (i.e. $(iB)^H = -iB^H = iB$). Moreover we have that
$$(iB)^{-1}A^H(iB) = -iB^{-1}A^H(iB) = -i^2 B^{-1}A^HB = B^{-1}A^HB.$$
This shows that $A \in \mathbb{J}(B)$ ($A \in \mathbb{L}(B)$, resp.) if and only if $A \in \mathbb{J}(iB)$ ($A \in \mathbb{L}(iB)$, resp.). Moreover
$$A^H(iB)A=iB \; \Leftrightarrow \; i \big(A^HBA \big) = iB \; \Leftrightarrow \; A^HBA = B,$$
so $A \in \mathbb{G}(B)$ if and only if $A \in \mathbb{G}(iB)$. Finally, the same reasoning reveals that $\mathcal{N}(B) = \mathcal{N}(iB)$. Therefore, some matrix in any of these sets corresponding to $B$ can always be interpreted as a matrix from the same set corresponding to $iB$. The main theorems obtained in the previous sections thus apply directly when $B=-B^H$.

\begin{theorem} \label{thm:main3}
Let $B \in \Gl_n(\CC)$ with $B=-B^H$.
For any $A \in \MM_n(\CC)$ in any of the sets $\mathbb{J}(B), \mathbb{L}(B), \mathbb{G}(B), \mathcal{N}(B)$ and any $\varepsilon > 0$ there exists some diagonalizable $A' \in \MM_n(\CC)$ belonging to the same set such that $\Vert A-A' \Vert_2 \leq \varepsilon$. In addition, the set of matrices with pairwise distinct eigenvalues is dense in $\mathbb{J}(B)$, $\mathbb{L}(B)$ and $\mathbb{G}(B)$.
\end{theorem}

\begin{proof}
Any matrix $A \in \MM_n(\CC)$ from $\mathbb{J}(B), \mathbb{L}(B)$, $\mathbb{G}(B)$ or $\mathcal{N}(B)$ can be interpreted as a matrix from $\mathbb{J}(iB), \mathbb{L}(iB)$, $\mathbb{G}(iB)$ or $\mathcal{N}(iB)$, respectively, and Theorems \ref{thm:dense-subspaces}, \ref{thm:dense-automorph} and \ref{thm:main2} apply.
\end{proof}

Certainly, there are analogous results for the case $B=-B^H$ as stated in Corollary \ref{cor:sumoftwo} and Proposition \ref{prop:sumoffour}.

\section{Conclusions}
\label{s:conclusions}

In this work we considered the structure classes of $B$-selfadjoint, $B$-skew\-adjoint, $B$-unitary and $B$-normal matrices defined by an (indefinite) scalar product $[x,y] = x^HBy$ on $\CC^n \times \CC^n$  for some $B \in \Gl_n(\CC)$. We showed that, if $B= \pm B^H$, the set of semisimple (i.e. diagonalizable) matrices is dense in the set of all $B$-selfadjoint, $B$-skewadjoint, $B$-unitary and $B$-normal matrices.

\end{document}